\documentclass[a4paper,10pt]{article}
\usepackage[dvipdfmx]{graphicx}
\usepackage{float}
\usepackage{latexsym,amssymb,amsmath,mathrsfs}
\usepackage{amsthm}
\usepackage{mathrsfs}
\usepackage{color}
\usepackage[english]{babel}
\newtheorem{theorem}{Theorem}
\newtheorem*{theorem*}{Theorem}
\newtheorem{remark}{Remark}
\newtheorem*{remark*}{Remark}
\newtheorem{lemma}{Lemma}
\newtheorem*{lemma*}{Lemma}

\newtheorem*{example*}{Example}

\newtheorem*{proposition*}{Prop}

\newtheorem*{question*}{Quest}

\newtheorem*{assume*}{Assume}

\newtheorem*{claim*}{claim}
\newtheorem{corollary}{Corollary}
\newtheorem*{corollary*}{Corollary}
\DeclareMathOperator{\Ric}{Ric}

\newcommand{\R}{\mathbb{R}}

\newcommand{\ez}{\varepsilon}
\newcommand{\Lap}{\Delta}
\newcommand{\bs}{\ensuremath{\mathbf{s}}}
\newcommand{\e}{\mathrm{e}}

\renewcommand{\d}{\mathrm{d}}

\newcommand{\eps}{{\varepsilon}}

\def\Cut{\mathop{\mathrm{Cut}}\nolimits}
\title{Aronson-B\'{e}nilan gradient estimates for porous medium equations under lower bounds of $N$-weighted Ricci curvature with $N < 0$}
\author{Yasuaki Fujitani\thanks{Department of Mathematics, Osaka University, Osaka 560-0043, Japan (\texttt{u197830k@ecs.osaka-u.ac.jp})}}
\begin{document}
\maketitle
\begin{abstract}
    The Aronson-B\'{e}nilan gradient estimate for the porous medium equation has been studied as a counterpart to the Li-Yau gradient estimate for the heat equation. 
    In this paper, we give the Aronson-B\'{e}nilan gradient estimates for the porous medium equation on weighted Riemannian manifolds under lower bounds of $N$-weighted Ricci curvature with $\eps$-range for some $N < 0$. 
    This is a generalization of those estimates under constant lower $N$-weighted Ricci curvature bounds with $N\in [n,\infty)$.
\end{abstract}
\tableofcontents
\section{Introduction}
The porous medium equation $\partial_t u = \Delta_\psi u^m$ $(m > 1)$ is a nonlinear deformation of the heat equation and has been studied intensively from both analytic and geometric viewpoints (see \cite{vazque1}).
One of the pioneering works on the porous medium equation is the Aronson-B\'{e}nilan estimate on Euclidean space in 1979 \cite{Aronson}.
This estimate was extended to compact Riemannian manifolds in \cite{vazque1}. 
Then, the local Aronson-B\'{e}nilan gradient estimate on Riemannian manifolds was obtained in \cite{villani} by extending the Li-Yau gradient estimate for the heat equation under lower bounds of the Ricci curvature.
The local estimate in \cite{villani} was used to generalize the Aronson-B\'{e}nilan estimate in \cite{vazque1} to noncompact complete Riemannian manifolds and they were improved by \cite{huang-huang-li}.
The local Aronson-B\'{e}nilan gradient estimates are closely related to comparison geometry concerning the lower bounds of the Ricci curvature in that those estimates are proved using the Laplacian comparison theorem. 

This paper is concerned with the weighted Ricci curvature $\Ric_\psi^N$, which is a generalization of the Ricci curvature for weighted Riemannian manifolds. Here, $N$ is a real parameter called the effective dimension and $\psi$ is a weight function.
Many comparison geometric results similar to those for Riemannian manifolds with Ricci curvature bounded from below by $K$ and dimension bounded from above by $N$ also hold if we assume the condition $\Ric_\psi^N\geq K$ for $K\in\mathbb{R}$.
Especially the case of $N\geq n$ is well investigated. 
For example, the local Aronson-B\'{e}nilan gradient estimates on Riemannian manifolds with $\Ric_g \geq -K$ for $K\geq 0$ in \cite{huang-huang-li} were generalized to weighted Riemannian manifolds with $\Ric_\psi^N \geq -K$ for $K\geq 0$ and $N \geq n$ in \cite{huang}. 
We also refer to \cite{hqiu} for a further generalization to the $V$-Laplacian.

Recently, the $N$-weighted Ricci curvature in the case of $N \in (-\infty,1]$ is getting more and more attention.
For this range, some Poincar\'{e} inequalities \cite{milman} (we refer to \cite{mai} for its rigidity), Beckner inequality \cite{gentil} and the curvature-dimension condition \cite{ohta_negative} were established. 
Contrary to the case $\Ric_\psi^N\geq K$ for $N\in[n,\infty)$, it is known that some comparison theorems (such as the Bishop-Gromov volume comparison theorem and the Laplacian comparison theorem) do not hold under the constant curvature bound $\Ric_\psi^N\geq K$ for $N\in(-\infty,1]\cup \{\infty\}$.
Nonetheless, Wylie-Yeroshkin \cite{yoroshikin} introduced a variable curvature bound 
\begin{equation*}
    \Ric_\psi^1 \geq K\e^{-\frac{4}{n-1}\psi}g
\end{equation*}
associated with the weight function $\psi$, and established several comparison theorems. They were then generalized to 
\begin{equation*}
    \Ric_\psi^N\geq K\e^{\frac{4}{N-n}\psi}g
\end{equation*}
with $N\in (-\infty, 1)$ by Kuwae-Li \cite{kuwae_li}. Furthermore, Lu-Minguzzi-Ohta \cite{lu2} gave a generalization of them as
\begin{equation*}
    \Ric_\psi^N\geq K\e^{\frac{4(\eps-1)}{n-1}\psi}g
\end{equation*}
with an additional parameter $\eps$ in an appropriate range, depending on $N$, called the $\eps$-range (we also refer to \cite{lu} for a preceding work on singularity theorems in Lorentz-Finsler geometry).
This is not only a generalization of \cite{yoroshikin} and \cite{kuwae_li}, but also an integration of both constant and variable curvature bounds by choosing appropriate $\eps$.
We refer to \cite{kuwae_sakurai,kuwae_sakurai_2,kuwae_sakurai_3,yasuaki} for further works on the $\eps$-range.

{In this paper, we study gradient estimates for the porous medium equation under lower bounds of $N$-weighted Ricci curvature with $N < 0$. In the region where $N$ is negative, we give a local Aronson-B\'{e}nilan gradient estimate in terms of $\eps$-range and the Aronson-B\'{e}nilan estimate on compact Riemannian manifolds with non-negative weighted Ricci curvature. 
Our results are generalizations of those estimates in \cite{huang,lili-2} for $N\in[n,\infty)$ to $ N\in (-\infty, -\frac{2}{m-1})\cup \{\infty\}$.}

The Aronson-B\'{e}nilan estimate under lower bounds of $N$-weighted Ricci curvature with $N < 0$
offers a new perspective to the Li-Yau gradient estimate for the heat equation under lower bounds of $N$-weighted Ricci curvature with $N < 0$, which is known as an open question as mentioned in \cite[Section 1 (b)]{ohta_negative}, \cite[Remark 14.16]{ohta-finsler}.
Letting $m \searrow 1$, the porous medium equation becomes the heat equation and the local Aronson-B\'{e}nilan gradient estimate recovers the Li-Yau gradient estimate under $\Ric_\psi^N \geq -K$ ($K\geq 0, N \geq n$).
However, our local gradient estimate is available for $N \in (-\infty,-\frac{2}{m-1})$ and, as $m\searrow 1$, $-\frac{2}{m-1}$ goes to $-\infty$ and the region $(-\infty, -\frac{2}{m-1})$ degenerates.
From this observation, the usual Li-Yau gradient estimate may not extend to $N < 0$ and we need an alternative approach (see also Remark \ref{fast-diffusion}).

This paper is organized as follows. In Section 2, we briefly review the weighted Ricci curvature with $\eps$-range and the Aronson-B\'{e}nilan gradient estimates. In Section 3, we give the local Aronson-B\'{e}nilan gradient estimate under lower bounds of the weighted Ricci curvature with $\eps$-range.
In Section 4, we give the Aronson-B\'{e}nilan estimate on compact Riemannian manifolds with non-negative weighted Ricci curvature.
\section{Preliminaries}
\subsection{$\varepsilon$-range}
Let $(M,g,\mu)$ be an $n$-dimensional weighted Riemannian manifold.
We set $\mu = \e^{-\psi}v_g$ where $v_g$ is the Riemannian volume measure and $\psi$ is a $C^{\infty}$ function on $M$.
For $N\in(-\infty,1]\cup[n,+\infty]$, the \emph{$N$-weighted Ricci curvature} is defined as follows:
\begin{align*}
\Ric_\psi^N:=\Ric_g+{\rm \nabla^2}\psi-\frac{\d\psi\otimes \d\psi}{N-n},
\end{align*}
where when $N=+\infty$,
the last term is interpreted as the limit $0$
and when $N=n$,
we only consider a constant function $\psi$, and
set $\Ric_\psi^n:=\Ric_g$.
 
In \cite{lu2}, they introduced the notion of \textit{$\eps$-range}:
\begin{equation}\label{epsilin-range}
\eps = 0 \mbox{ for } N = 1, \quad |\eps| < \sqrt{\frac{N-1}{N-n}} \mbox{ for } N\neq 1,n, \quad \eps\in \mathbb{R}\mbox{ for } N = n.
\end{equation}
In this $\eps$-range, for $K\in\mathbb{R}$, they considered the condition
\[ \Ric_\psi^N(v)\ge K \e^{\frac{4(\ez-1)}{n-1}\psi(x)} g(v,v),\quad v\in T_xM.\]
We also define a constant $c$ associated with $\eps$ as
\begin{equation}\label{c_num}
c = \frac{1}{n-1}\left(1- \eps^2\frac{N-n}{N-1}\right) > 0
\end{equation}
for $N\neq 1$ and $c = (n-1)^{-1}$ for $N = 1$.
We define the comparison function $\bs_{\kappa}$ as
\begin{equation}\label{eq:bs}
\bs_{\kappa}(t) := \begin{cases}
\frac{1}{\sqrt{\kappa}} \sin(\sqrt{\kappa}t) & \kappa>0, \\
t & \kappa=0, \\
\frac{1}{\sqrt{-\kappa}} \sinh(\sqrt{-\kappa}t) & \kappa<0.
\end{cases}
\end{equation}
We denote $B(x,r)=\{ y \in M \,|\, d(x,y)<r \}$. 
The \emph{weighted Laplacian} $\Delta_\psi$ on $(M,g,\mu)$ is defined as 
\begin{equation*}
    \Delta_\psi = \Delta - \langle \nabla\psi,\cdot \rangle.
\end{equation*}
\begin{theorem}(\cite[Theorem 3.9]{lu2}, Laplacian comparison theorem)\label{laplacian_comparison_theorem}
    Let $(M, g, \mu)$ be an $n$-dimensional complete weighted Riemannian manifold and $N \in (-\infty,1] \cup [n,+\infty]$,
    $\ez \in \R$ in the $\ez$-range \eqref{epsilin-range}, $K \in \R$ and $p_2 \ge p_1>0$.
    Assume that
    \[ \Ric_\psi^N(v)\ge K \e^{\frac{4(\ez-1)}{n-1}\psi(x)} g(v,v)\]
    holds for all $v \in T_xM \setminus 0$ and
    \[ p_1 \le \e^{-\frac{2(\ez-1)}{n-1}\psi} \le p_2. \]
    Then, for any $z \in M$, the distance function $r(x):=d(z,x)$ satisfies
    \[ \Lap_{\psi} r(x) \le \frac{1}{c\rho} \frac{\bs'_{cK}(r(x)/p_2)}{\bs_{cK}(r(x)/p_2)} \]
    on $M \setminus (\{z\} \cup \Cut(z))$, where $\rho:=p_1$ if $\bs'_{cK}(r(x)/p_2) \ge 0$
    and $\rho:=p_2$ if $\bs'_{cK}(r(x)/p_2)<0$, and $\Cut(z)$ denotes the cut locus of $z$.
\end{theorem}

\subsection{Aronson-B\'{e}nilan estimates for porous medium equation}
The Aronson-B\'{e}nilan estimate is closely related to the Li-Yau estimates for the heat equation:
\begin{equation}\label{heat-eq}
    \partial_t u = \Delta u.
\end{equation}
The following theorem is the local Li-Yau gradient estimate.
\begin{theorem}(Local Li-Yau gradient estimate, \cite{li-yau}). Let $\left(M, g\right)$ be an $n$-dimensional complete Riemannian manifold with $\Ric_g \geq-K$ on $B(p,2R)$ with $K\geq 0$. Suppose that $u$ is a positive 
    solution to the heat equation \eqref{heat-eq} on $B(p,2R) \times[0, T]$. Then for any $\alpha>1$, we have
    \begin{equation}\label{li-yau-gradient-estimate}
        \frac{|\nabla u|^2}{u^2}-\alpha \frac{\partial_t u}{u} \leq \frac{C \alpha^2}{R^2}\left(\frac{\alpha^2}{\alpha-1}+\sqrt{K} R\right)+\frac{n \alpha^2 K}{2(\alpha-1)}+\frac{n \alpha^2}{2 t}
    \end{equation}
    on $B(p,R)\times (0,T]$, where $C$ is a constant depending on $n$.
\end{theorem}
Moreover, letting $R\rightarrow \infty$ in \eqref{li-yau-gradient-estimate}, we have the following Li-Yau estimate.
\begin{corollary}(Li-Yau estimate, \cite{li-yau})
    Let $(M,g)$ be an $n$-dimensional complete noncompact Riemannian manifold with $\operatorname{Ric}_g \geq-K$ with $K \geq 0$. Suppose that $u$ is a positive solution to the heat equation \eqref{heat-eq} on $M\times [0,T]$. Then for any $\alpha > 1$, we have
\begin{equation}\label{li-yau-estimate-alpha}
    \frac{|\nabla u|^2}{u^2}-\alpha \frac{\partial_t u}{u} \leq \frac{n \alpha^2 K}{2(\alpha-1)}+\frac{n \alpha^2}{2 t} 
\end{equation}
on $M\times (0,T]$.
In the special case where $K=0$, there holds 
\begin{equation}\label{li-yau-estimate}
    \frac{|\nabla u|^2}{u^2}-\frac{\partial_t u}{u} \leq \frac{n}{2 t}
\end{equation}
on $M\times (0,T]$.
\end{corollary}
As a pioneering work, for any positive solution $u$ to $\partial_t u = \Delta u^m$ on $\mathbb{R}^n$ with $m > 1-\frac{2}{n}$, Aronson-Bénilan \cite{Aronson} derived a celebrated second order differential inequality
\begin{equation}\label{aronson-classical}
\sum_{i=1}^n \frac{\partial}{\partial x^i}\left(m u^{m-2} \frac{\partial u}{\partial x^i}\right) \geq-\frac{\kappa}{t},
\end{equation}
where $\kappa=\frac{n}{n(m-1)+2}$. When $m=1$, the above inequality \eqref{aronson-classical} recovers the classical Li-Yau estimate \eqref{li-yau-estimate}.

This paper is concerned with the porous medium equation for the weighted Laplacian:
\begin{equation}\label{porus_medium_equation}
    \frac{\partial u}{\partial t} = \Delta_\psi u^m\quad (m > 1)
\end{equation}
on $M\times [0,T]$.
The following theorem is a local Aronson-B\'{e}nilan gradient estimate for the porous medium equation under constant lower bounds of the weighted Ricci curvature; we refer to \cite[Theorem 1.1]{huang-huang-li} for the unweighted case and to \cite[Theorem 3.1]{huang}, \cite[Theorem 1]{hqiu} for the weighted case. We generalize the following theorem to the $\eps$-range in Section 3.
\begin{theorem}(Local Aronson-B\'{e}nilan  gradient estimate)\label{aronson-benilan-huang}
    Let $(M,g,\mu)$ be an $n$-dimensional complete weighted Riemannian manifold  with $\Ric_\psi^N \geq$ $-K$ on $B(p,2 R)$ with $K \geq 0$ and $N \geq n$. Suppose that $u$ is a positive solution of \eqref{porus_medium_equation} with $m > 1$ on $M\times [0,T]$. Let $v=\frac{m}{m-1} u^{m-1}$ and $L=(m-1)\sup_{B(p,2 R) \times[0, T]} v$. Then for any $\alpha>1$, we have
    \begin{eqnarray}\label{huang-li-estimate-1}
   \frac{|\nabla v|^2}{v}-\alpha \frac{\partial_t v}{v} &\leq& \Biggl\{\frac{a \alpha^2 m L^{1/2}}{(\alpha-1)^{1/2}(m-1)^{1/2}} \frac{C}{R}\nonumber\\
   &&\qquad+a^{1/2} \alpha\left[\frac{1}{t}+\frac{L K}{2(\alpha-1)}+ \frac{CL}{R^2}\left\{1+\sqrt{K}R \operatorname{coth}\left(\sqrt{\frac{K}{N-1}} R\right)\right\}\right]^{1/2}\Biggr\}^2
\end{eqnarray}
on $B(p,R)\times (0,T]$,
   where $a=\frac{N(m-1)}{N(m-1)+2}$ and $C$ is a constant depending on $N$.
   \end{theorem}
   We remark that, although in the estimates in \cite{huang-huang-li,huang,hqiu}, we have $\sqrt{K}R$ inside $\coth$, from its proof, the above inequality seems the correct one (compare \cite[(2.12)]{huang} and the Laplacian comparison estimate 2 lines before).

This result gives an improvement of the local Aronson-B\'{e}nilan gradient estimate in \cite{villani}. 
We note that letting $R\rightarrow \infty$ in \eqref{huang-li-estimate-1}, we obtain the following Aronson-B\'{e}nilan estimate for the porous medium equation; we refer to \cite[Corollary 1.1]{huang-huang-li} for the unweighted case and to \cite[Theorem 3.1]{huang}, \cite[Corollary 1]{hqiu} for the weighted case.
\begin{corollary}(Aronson-B\'{e}nilan estimate)\label{aronson-benilan-estimate}
     Let $\left(M, g,\mu\right)$ be an $n$-dimensional complete noncompact weighted Riemannian manifold with $\Ric_\psi^N \geq-K$, where $K \geq 0$ and $N \geq n$. Suppose that $u$ is a positive solution of \eqref{porus_medium_equation} with $m > 1$ on $M\times [0,T]$. Let $v=\frac{m}{m-1} u^{m-1}$ and $L=(m-1)\sup_{M \times[0, T]} v$. Then for any $\alpha>1$, we have
    $$
    \frac{|\nabla v|^2}{v}-\alpha \frac{\partial_t v}{v} \leq a \alpha^2\left(\frac{1}{t}+\frac{K L}{2(\alpha-1)}\right) 
    $$
    on $M\times (0,T]$, where $a = \frac{N(m-1)}{N(m-1) + 2}$. In the special case where $K=0$, we have 
    \begin{equation}\label{classical-ab-estimate}
        \frac{|\nabla v|^2}{v} - \frac{\partial_t v}{v}\leq \frac{a}{t}
    \end{equation}
    on $M\times (0,T]$.
\end{corollary}
We remark that \eqref{classical-ab-estimate} recovers the classical Aronson-B\'{e}nilan estimate \eqref{aronson-classical} (we also refer to \cite[Corollary 3.4]{villani}, where they generalized the Aronson-B\'{e}nilan estimate on compact Riemannian manifolds in \cite[Proposition 11.12]{vazque1} to noncompact complete manifolds). 
We refer to \cite{hqiu} and references therein for more on the background.

\section{Aronson-B\'{e}nilan estimate under lower weighted Ricci curvature bounds with $\eps$-range}
In this section, we prove a local Aronson-B\'{e}nilan gradient estimate under lower bounds of the weighted Ricci curvature with $\eps$-range.
Let $u$ be a positive solution of \eqref{porus_medium_equation} with $m > 1$. 
We set 
\begin{equation*}
    v = \frac{m}{m-1}u^{m-1}.
\end{equation*}
From direct calculations, we have
\begin{equation}
    \partial_t v = |\nabla v|^2 + (m-1)v\Delta_\psi v\label{moto-2-1}.
\end{equation}
We define
\begin{equation*}
    F_\alpha=\frac{|\nabla v|^{2}}{v}-\alpha \frac{\partial_t v}{v}
\end{equation*}
for a constant $\alpha \in \mathbb{R}$ and 
\begin{equation*}
    \mathscr{L}:=\frac{\partial}{\partial t}-(m-1) v \Delta_{\psi} .
\end{equation*}
We set a positive constant depending on $m$ and $N$:
 \begin{equation}\label{a-def}
    a(m,N) = \begin{cases}
        1 & \mbox{if } N = \infty,\\
        \frac{N(m-1)}{N(m-1)+2} & \mbox{if } N \in (-\infty,\frac{-2}{m-1}) \cup [n,+\infty).
    \end{cases}
\end{equation}
 We prove the following theorem.
\begin{theorem}\label{aronson-benilan-epsilon-range}
    Let $(M,g,\mu)$ be an $n$-dimensional complete weighted Riemannian manifold, $u$ be a positive solution of \eqref{porus_medium_equation} with $m > 1$ on $M\times [0,T]$ and 
    $N \in (-\infty,\frac{-2}{m-1}) \cup [n,+\infty]$. 
    We assume $\Ric_\psi ^N \geq -K \e^{\frac{4(\ez-1)}{n-1}\psi}$ on $B(p, 2R)$ with $\ez \in \R$ in the $\ez$-range \eqref{epsilin-range}, $K \geq 0$
    and
    \[ p_1 \le \e^{-\frac{2(\ez-1)}{n-1}\psi} \le p_2 \] 
    with $p_2 \ge p_1>0$ in $B(p,2R)$ for some fixed point $p$ in $M$. 
    Let $v = \frac{m}{m-1}u^{m-1}$ and $L = (m-1)\sup_{B(p,2R)\times[0,T]}v$. Then for any $\alpha > 1$, we have 
    \begin{eqnarray*}
            \frac{|\nabla v|^{2}}{v}-\alpha \frac{\partial_t v}{v}  &\leq& \Bigg\{\frac{a\alpha^2 m L^{1/2}}{(\alpha-1)^{1/2}(m-1)^{1/2}} \frac{C}{R}\\
            &&+ a^{1/2}\alpha \left[\frac{1}{t}+\frac{K}{p_1^2}\frac{L}{2(\alpha-1)} +\frac{C L}{R^{2}}\left\{1+\sqrt{K} R \operatorname{coth}\left(\frac{\sqrt{cK}}{p_2} R\right)\right\}\right]^{1 / 2}\Bigg\}^{2}
        \end{eqnarray*}
        on $B(p,R)\times (0,T]$, where $a = a(m,N)$ defined in \eqref{a-def} and $C$ is a constant depending on $N,\eps, p_1$.
\end{theorem}
We first prove the following lemma:
\begin{lemma}\label{basic_lemma_porus_medium}
    Let $(M,g,\mu)$ be an $n$-dimensional complete weighted Riemannian manifold, $u$ be a positive solution of \eqref{porus_medium_equation} with $m > 1$ on $M\times[0,T]$ and $N \in (-\infty,\frac{-2}{m-1}) \cup [n,+\infty]$. Then we have
    \begin{eqnarray}
        \mathscr{L}(F_\alpha) &\leq &-\frac{1}{a}\left[(m-1) \Delta_{\psi} v\right]^{2}+(1-\alpha)\left(\frac{\partial_t v}{v}\right)^{2}+2 m\langle\nabla v, \nabla F_\alpha\rangle  -2(m-1) \operatorname{Ric}_{\psi}^{N}( \nabla v)\label{basic_ineq_porus_medium},
    \end{eqnarray}
    where $v = \frac{m}{m-1}u^{m-1}$ and $a = a(m,N)$ defined in \eqref{a-def}.
\end{lemma}
\begin{proof}
    This is shown from the $N$-weighted Bochner inequality:
    \begin{equation}\label{weighted_bochner_N}
        \Delta_{\psi}\left(\frac{|\nabla u|^2}{2}\right)-\left\langle\nabla \Delta_{\psi} u, \nabla u\right\rangle  \geq \operatorname{Ric}_\psi^N(\nabla u)+\frac{\left(\Delta_{\psi} u\right)^2}{N}
    \end{equation}
    for $u \in {C}^{\infty}(M)$, which is valid also for $N < 0$ as is proved in \cite[Theorem 4.1]{ohta_negative}. 
    When $ N= \infty$, we have 
    \begin{equation}\label{weighted_bochner_infty}
        \Delta_{\psi}\left(\frac{|\nabla u|^2}{2}\right)-\left\langle\nabla \Delta_{\psi} u, \nabla u\right\rangle = \operatorname{Ric}_\psi^\infty(\nabla u)+\|\nabla^2 u\|_{HS}^2\geq \operatorname{Ric}_\psi^\infty  (\nabla u).
    \end{equation}
    We first prove \eqref{basic_ineq_porus_medium} when $N \in (-\infty,\frac{-2}{m-1})\cup [n,\infty)$. 
 As is shown in \cite[Proposition 7.5]{lili-2}, we have 
 \begin{equation}\label{lili-2-bacic-equality}
    \mathscr{L}(F_\alpha) = -2(m-1)\left(\|\nabla^2 v\|_{HS}^2 + \Ric_\psi^{\infty}(\nabla v)\right) + 2m\langle\nabla F_\alpha,\nabla v \rangle -(\alpha-1)\left(\frac{\partial_t v}{v}\right)^2 - F_1^2 .
 \end{equation}
 Combining \eqref{weighted_bochner_N} and \eqref{weighted_bochner_infty} with \eqref{lili-2-bacic-equality}, we have 
 \begin{eqnarray*}
    \mathscr{L}(F_\alpha) &\leq& -2(m-1)\left(\frac{(\Delta_\psi v)^2}{N} + \Ric^N_\psi(\nabla v)\right) + 2m\langle \nabla F_\alpha, \nabla v \rangle - (\alpha-1)\left(\frac{\partial_t v}{v}\right)^2 - F_1^2\\
    &=& \left\{ \frac{-2(m-1)}{N} - (m-1)^2\right\}(\Delta_\psi v)^2 -2(m-1)\Ric_{\psi}^N(\nabla v) + 2m\langle \nabla F_\alpha,\nabla v \rangle - (\alpha-1)\left(\frac{\partial_t v}{v}\right)^2,
 \end{eqnarray*}
 where we used $F_1 = -(m-1)\Delta_\psi v$.
Therefore, we obtain the desired inequality. The case of $N = \infty$ follows by the same method.
\end{proof}

\underline{\textit{Proof of Theorem \ref{aronson-benilan-epsilon-range}}}

We apply the argument in \cite[Theorem 1]{hqiu}, \cite[Theorem 1.1]{huang}.
First, we construct a suitable cut-off function $\phi$ (we also refer to \cite[Theorem 2.1]{wu} for the construction of $\phi$).
Let $\eta$ be a non-negative cut-off function on $[0, \infty)$ such that
$$
\eta(r)= \begin{cases}1 & \text { on }[0,1], \\ 0 & \text { on }[2, \infty)\end{cases}
$$
and $0 \leq \eta \leq 1$ on the interval $(1,2)$ with $-C_0 \eta^{1 / 2}(r) \leq \eta^{\prime}(r) \leq 0$ and $\eta^{\prime \prime}(r) \geq-C_0$, where $C_0 > 0$ is a universal constant. Let $R>0$ and $r(x):=d(p, x)$. We define
$$
\phi(x):=\eta\left(\frac{r(x)}{R}\right) .
$$
By construction, we have
$$
\frac{|\nabla \phi|^2}{\phi} \leq \frac{C_1}{R^2},
$$
where $C_1$ is a constant depending only on $C_0$.
We first consider the case of $K > 0$. By Theorem \ref{laplacian_comparison_theorem}, we have
\begin{eqnarray*}
    \Delta_\psi r\leq \frac{\sqrt{K}}{p_1\sqrt{c}}\coth\left(\frac{\sqrt{cK}}{p_2}r\right).
\end{eqnarray*}
Therefore, we obtain
\begin{eqnarray}
    \Delta_\psi \phi&=&\frac{\eta^{\prime}(r/R) \Delta_\psi r}{R}+\frac{\eta^{\prime \prime}(r/R)|\nabla r|^2}{R^2}\nonumber\\
    &\geq& \frac{\eta^{\prime}(r/R)}{R}\frac{\sqrt{K}}{p_1\sqrt{c}}\coth\left(\frac{\sqrt{cK}}{p_2}r\right) - \frac{C_0}{R^2}\nonumber\\
    &\geq& \frac{\eta^{\prime}(r/R)}{R}\frac{\sqrt{K}}{p_1\sqrt{c}}\coth\left(\frac{\sqrt{cK}}{p_2}R\right)- \frac{C_0}{R^2}\nonumber\\
    &\geq& -\frac{C_0\eta^{1/2}(r/R)}{R}\frac{\sqrt{K}}{p_1\sqrt{c}}\coth\left(\frac{\sqrt{cK}}{p_2}R\right)- \frac{C_0}{R^2}\nonumber\\
    &\geq& -\frac{C_2}{R^2}\left\{1 + \sqrt{K}R \coth\left(\frac{\sqrt{cK}}{p_2}R\right)\right\},\label{comparison-cut-off}
\end{eqnarray}
where we used $\eta^{\prime}(r(x) / R) = 0$ if $r(x) < R$ in the second inequality and $C_2$ is a constant depending on $N, \eps, p_1$.  
In the following argument, we fix $\alpha \geq 1$ and set $\widetilde{A} = K/p_1^2$ and $G = t\phi F_\alpha$.

Substituting $\operatorname{Ric}_\psi^N \geq-\widetilde{A}$ to \eqref{basic_ineq_porus_medium}, we obtain
$$
\mathscr{L}(F_\alpha) \leq-\frac{1}{a}\left[(m-1) \Delta_\psi v\right]^2+2 m\langle\nabla v, \nabla F_\alpha\rangle+2 \widetilde{A} L \frac{|\nabla v|^2}{v}
$$
on $B(p,2R)\times [0,T]$.
Since
$$
(m-1) \Delta_\psi v=\frac{\partial_t v}{v}-\frac{|\nabla v|^2}{v}=\frac{1}{\alpha}\left(\frac{|\nabla v|^2}{v}-F_\alpha\right)-\frac{|\nabla v|^2}{v}=-\frac{1}{\alpha} F_\alpha+ \frac{1-\alpha}{\alpha} \frac{|\nabla v|^2}{v} ,
$$
we have
$$
\begin{aligned}
\mathscr{L}(F_\alpha) & \leq -\frac{1}{a \alpha^2}\left(F_\alpha+(\alpha-1) \frac{|\nabla v|^2}{v}\right)^2+2 m\langle\nabla v, \nabla F_\alpha\rangle+2 \widetilde{A} L \frac{|\nabla v|^2}{v} .
\end{aligned}
$$
We apply the maximum principle to $G$ on $B(p,2 R) \times[0, T]$. 
Let $(x_1,t_1)\in B(p,2 R) \times[0, T]$ be a point that attains the maximum of $G$ and we assume $G\left(x_1, t_1\right)>0$ (otherwise the proof is trivial), which implies $t_1>0$. 
Since 
\begin{eqnarray*}
    \nabla G(x_1,t_1) = 0,\quad \Delta G(x_1,t_1) \leq 0,\quad \partial_t G(x_1,t_1) \geq 0,
\end{eqnarray*}
we find
$$
\mathscr{L}(G)(x_1,t_1)  \geq 0.
$$
Therefore, we have at $(x_1, t_1)$,
$$
\begin{aligned}
0 &\leq \mathscr{L}(G) =\left(\frac{\partial}{\partial t}-(m-1) v \Delta_\psi\right)(t \phi F_\alpha)\\
&=\phi F_\alpha+t_1 \phi \frac{\partial}{\partial t} F_\alpha-(m-1) v \Delta_\psi\left(t_1 \phi F_\alpha\right) \\
&=\phi F_\alpha+t_1 \phi \frac{\partial}{\partial t} F_\alpha-(m-1) v t_1\left(F_\alpha \Delta_\psi \phi+2\langle\nabla \phi, \nabla F_\alpha\rangle+\phi \Delta_\psi F_\alpha\right) \\
&=\phi F_\alpha+t_1 \phi \mathscr{L} (F_\alpha)-(m-1) v t_1 F_\alpha \Delta_\psi \phi-2(m-1) v t_1\langle\nabla \phi, \nabla F_\alpha\rangle \\
&=\frac{G}{t_1}+t_1 \phi \mathscr{L} (F_\alpha)-(m-1) v \frac{\Delta_\psi \phi}{\phi} G+2(m-1) v \frac{|\nabla \phi|^2}{\phi^2} G\\
&\leq  \frac{G}{t_1}+t_1 \phi\left\{-\frac{1}{a \alpha^2}\left(F_\alpha+(\alpha-1) \frac{|\nabla v|^2}{v}\right)^2+2 m\langle\nabla v, \nabla F_\alpha\rangle+2 \widetilde{A} L \frac{|\nabla v|^2}{v}\right\} \\
&\quad -(m-1) v \frac{\Delta_\psi \phi}{\phi} G+2(m-1) v \frac{|\nabla \phi|^2}{\phi^2} G \\
&= \frac{G}{t_1}+t_1 \phi\left\{-\frac{1}{a \alpha^2}\left(F_\alpha+(\alpha-1) \frac{|\nabla v|^2}{v}\right)^2-\frac{2 m F_\alpha}{\phi}\langle\nabla v, \nabla \phi\rangle+2 \widetilde{A} L \frac{|\nabla v|^2}{v}\right\} \\
&\quad -(m-1) v \frac{\Delta_\psi \phi}{\phi} G+2(m-1) v \frac{|\nabla \phi|^2}{\phi^2} G ,
\end{aligned}
$$
where we used $\nabla G(x_1,t_1) = 0$ in the fifth and last equalities. 
Since $G(x_1,t_1) > 0$, we have $F_\alpha(x_1,t_1) > 0$. 
Let $\frac{|\nabla v|^2}{v}(x_1,t_1) =\beta F_\alpha(x_1,t_1)$ for some  $\beta \geq 0$.
It follows that
$$
\begin{aligned}
0 \leq & \frac{G}{t_1}+t_1 \phi\left\{-\frac{1}{a \alpha^2}(F_\alpha+(\alpha-1) \beta F_\alpha)^2-\frac{2 m F_\alpha}{\phi}\langle\nabla v, \nabla \phi\rangle+2 \widetilde{A} L \frac{|\nabla v|^2}{v}\right\} \\
&-(m-1) v \frac{\Delta_\psi \phi}{\phi} G+2(m-1) v \frac{|\nabla \phi|^2}{\phi^2} G \\
=& \frac{G}{t_1}-\frac{1}{a \alpha^2} t_1 \phi F_\alpha^2[1+(\alpha-1) \beta]^2-2 m t_1 \phi F_\alpha \frac{\langle\nabla v, \nabla \phi\rangle}{\phi}+2 \widetilde{A} L t_1 \phi \frac{|\nabla v|^2}{v} \\
&-(m-1) v \frac{\Delta_\psi \phi}{\phi} G+2(m-1) v \frac{|\nabla \phi|^2}{\phi^2} G \\
=& \frac{G}{t_1}-\frac{G^2}{a \alpha^2 t_1 \phi}[1+(\alpha-1) \beta]^2-2 m G \frac{\langle\nabla v, \nabla \phi\rangle}{\phi}+2 \widetilde{A} L \beta G \\
&-(m-1) v \frac{\Delta_\psi \phi}{\phi} G+2(m-1) v \frac{|\nabla \phi|^2}{\phi^2} G.
\end{aligned}
$$
By the Cauchy-Schwarz inequality, we have
$$
\begin{aligned}
\left|2 m G \frac{\langle\nabla v, \nabla \phi\rangle}{\phi}\right| & \leq 2 m G \frac{|\nabla v| \cdot|\nabla \phi|}{\phi}\\
&=2 m G \frac{|\nabla v|}{v^{1 / 2}} \cdot \frac{|\nabla \phi|}{\phi} \cdot v^{1 / 2}\\
&=2 m G \beta^{1 / 2} F_\alpha^{1 / 2} \frac{|\nabla \phi| \phi^{1 / 2} t_1^{1 / 2}}{(m-1)^{1 / 2} \phi^{3 / 2}t_1^{1 / 2}} (m-1)^{1 / 2} v^{1 / 2} \\
& \leq 2 m G \beta^{1 / 2} \frac{|\nabla \phi| G^{1 / 2}}{(m-1)^{1 / 2} \phi^{3 / 2} t_1^{1 / 2}} L^{1 / 2}.
\end{aligned}
$$
Therefore
$$
\begin{aligned}
0 \leq & \frac{G}{t_1}-\frac{G^2}{a \alpha^2 t_1 \phi}[1+(\alpha-1) \beta]^2+2 m G \beta^{1 / 2} L^{1 / 2} \frac{|\nabla \phi| G^{1 / 2}}{(m-1)^{1 / 2} \phi^{3 / 2} t_1^{1 / 2}}+2 \widetilde{A} L \beta G \\
&-(m-1) v \frac{\Delta_\psi \phi}{\phi} G+2(m-1) v \frac{|\nabla \phi|^2}{\phi^2} G.
\end{aligned}
$$
Multiplying both sides of the above inequality by $\frac{\phi}{G}$, we obtain
\begin{eqnarray}
    0 &\leq & \frac{\phi}{t_1}-\frac{G}{a \alpha^2 t_1}[1+(\alpha-1) \beta]^2+2 m \beta^{1 / 2} L^{1 / 2} \frac{|\nabla \phi| G^{1 / 2}}{(m-1)^{1 / 2} \phi^{1 / 2} t_1^{1 / 2}}+2 \widetilde{A} L \beta \phi -(m-1) v \Delta_\psi \phi+2(m-1) v \frac{|\nabla \phi|^2}{\phi}\nonumber.
\end{eqnarray}
Hence, we get
\begin{eqnarray}
    \frac{[1+(\alpha-1) \beta]^2}{a \alpha^2 t_1} G-2 m \frac{\beta^{1 / 2} L^{1 / 2}|\nabla \phi|}{(m-1)^{1 / 2} \phi^{1 / 2} t_1^{1 / 2}} G^{1 / 2}  \leq \frac{\phi}{t_1}+2 \widetilde{A} L \beta \phi  -(m-1) v \Delta_\psi \phi+2(m-1) v \frac{|\nabla \phi|^2}{\phi}.\qquad\qquad
\end{eqnarray}
Combining this with \eqref{comparison-cut-off}, we have 
\begin{eqnarray}\label{based_inequality}
    \frac{[1+(\alpha-1) \beta]^2}{a \alpha^2 t_1} G-2 m \frac{\beta^{1 / 2} L^{1 / 2}|\nabla \phi|}{(m-1)^{1 / 2} \phi^{1 / 2} t_1^{1 / 2}} G^{1 / 2}  &\leq& \frac{\phi}{t_1}+2 \widetilde{A} L \beta \phi  +L \frac{C_2}{R^2}\left\{1 + \sqrt{K}R\coth\left( \frac{\sqrt{cK}}{p_2}R\right)\right\} \nonumber\\
    &&\quad +2L \frac{|\nabla \phi|^2}{\phi} .
\end{eqnarray}

From the inequality $Ax^2 - Bx \leq C$ $(A > 0,\  B,C\geq 0)$, we have $x\leq \frac{B}{A} + \sqrt{\frac{C}{A}}$. Applying this with $x = G^{1/2}$, we obtain from \eqref{based_inequality} that
\begin{eqnarray}\label{moto-nashi-1}
    &&G^{1/2} \leq \frac{ 2a \alpha^2 m L^{1 / 2} t_1^{1 / 2}[(\alpha-1) \beta]^{1 / 2}}{(m-1)^{1 / 2}(\alpha-1)^{1 / 2}[1+(\alpha-1) \beta]^2} \frac{|\nabla \phi|}{\phi^{1 / 2}}\nonumber\\
    &&\quad+\left[\frac{a \alpha^2 t_1}{[1+(\alpha-1) \beta]^2}\left(\frac{\phi}{t_1}+2 \widetilde{A} L \beta \phi+  \frac{C_2L}{R^2}\left\{1 + \sqrt{K}R\coth\left( \frac{\sqrt{cK}}{p_2}R\right)\right\}+2L \frac{|\nabla \phi|^2}{\phi}\right)\right]^{1 / 2} .\qquad
\end{eqnarray}
Since
\begin{eqnarray*}
    \frac{[(\alpha-1) \beta]^{1 / 2}}{[1+(\alpha-1) \beta]^2} \leq \frac{1}{2}\frac{1+(\alpha-1) \beta}{[1+(\alpha-1) \beta]^2} \leq \frac{1}{2}
\end{eqnarray*}
and
$$
\begin{aligned}
\frac{a \alpha^2 t_1}{[1+(\alpha-1) \beta]^2}\left(\frac{1}{t_1}+2 \widetilde{A} L \beta \right) &=\frac{a \alpha^2}{[1+(\alpha-1) \beta]^2}+\frac{2 a \alpha^2 t_1 \widetilde{A} L \beta }{[1+(\alpha-1) \beta]^2} \\
& \leq a \alpha^2+\frac{2 a \alpha^2 t_1\widetilde{A} L}{\alpha-1} \frac{(\alpha-1) \beta}{[1+(\alpha-1) \beta]^2} \\
& \leq a \alpha^2\left(1+\frac{\widetilde{A} L t_1}{2(\alpha-1)}\right) ,
\end{aligned}
$$
we can estimate the right-hand side of \eqref{moto-nashi-1} as 
$$
\begin{aligned}
 G^{1/2}(x,T) \leq &G^{1 / 2}(x_1,t_1)\\
  \leq &\frac{a\alpha^2 m L^{1/2}t_1^{1/2}}{(\alpha-1)^{1/2}(m-1)^{1/2}} \cdot \frac{\sqrt{C_1}}{R}+\left\{a \alpha^2\left(1+\frac{\widetilde{A} L t_1}{2(\alpha-1)}\right)\right.\\
&\left.+\frac{a \alpha^2 t_1}{[1+(\alpha-1) \beta]^2}\left[\frac{C_2L}{R^2}\left\{1 + \sqrt{K}R \coth\left(\frac{\sqrt{cK}}{p_2}R\right)\right\}+\frac{2 C_1 L}{R^2}\right]\right\}^{1 / 2} \\
\leq & \frac{a\alpha^2 m L^{1/2}T^{1/2}}{(\alpha-1)^{1/2}(m-1)^{1/2}} \cdot \frac{\sqrt{C_1}}{R} \\
&+a^{1 / 2} T^{1 / 2} \alpha\left[\frac{1}{T}+\frac{\widetilde{A} L}{2(\alpha-1)}+\frac{(C_2 + 2C_1)L}{R^2}\left\{1 + \sqrt{K}R \coth\left(\frac{\sqrt{cK}}{p_2}R\right)\right\}\right]^{1 / 2} 
\end{aligned}
$$
for any $x \in B(p, R)$. Hence, we get
$$
\begin{aligned}
T^{1 / 2} F_\alpha(x, T)^{1 / 2} \leq & \frac{a\alpha^2 m L^{1/2}T^{1/2}}{(\alpha-1)^{1/2}(m-1)^{1/2}} \cdot \frac{\sqrt{C_1}}{R} \\
&+a^{1 / 2} T^{1 / 2} \alpha\left[\frac{1}{T}+\frac{\widetilde{A} L}{2(\alpha-1)}+\frac{C_3 L}{R^2}\left\{1 + \sqrt{K}R \coth\left(\frac{\sqrt{cK}}{p_2}R\right)\right\}\right]^{1 / 2} ,
\end{aligned}
$$
where $C_3 = 2C_1 + C_2$.
Dividing both sides by $T^{1/2}$, we have
$$
\begin{aligned}
F_\alpha(x, T) \leq & \Biggl\{\frac{a\alpha^2 m L^{1/2}}{(\alpha-1)^{1/2}(m-1)^{1/2}}  \frac{\sqrt{C_1}}{R}\\
&+ a^{1/2}\alpha \left[\frac{1}{T}+\frac{\widetilde{A} L}{2(\alpha-1)}+\frac{C_3L}{R^2}\left\{1 + \sqrt{K}R \coth\left(\frac{\sqrt{cK}}{p_2}R\right)\right\}\right]^{1 / 2}\Biggr\}^2 .
\end{aligned}
$$
Since $T$ is arbitrary, we obtain the desired inequality when $K > 0$. The case of $K = 0$ follows by taking the limit as $K\rightarrow 0$.
\qed

Letting $R\rightarrow \infty$ in Theorem \ref{aronson-benilan-epsilon-range}, we have the following estimate.
\begin{corollary}\label{ab-est-eps}
    Let $(M,g,\mu)$ be an $n$-dimensional complete non-compact weighted Riemannian manifold, $u$ be a positive solution of \eqref{porus_medium_equation} with $m > 1$ on $M\times [0,T]$ and 
    $N \in (-\infty,\frac{-2}{m-1}) \cup [n,+\infty]$. 
    We assume $\Ric_\psi ^N \geq -K \e^{\frac{4(\ez-1)}{n-1}\psi}$ on $M$ with $\ez \in \R$ in the $\ez$-range \eqref{epsilin-range}, $K \geq 0$
    and
    \[ p_1 \le \e^{-\frac{2(\ez-1)}{n-1}\psi} \le p_2 \] 
    with $p_2 \ge p_1>0$ in $M$. 
    Let $v = \frac{m}{m-1}u^{m-1}$ and $L = (m-1)\sup_{M\times[0,T]}v$. Then for any $\alpha > 1$, we have 
    \begin{equation*}
        \frac{|\nabla v|^2}{v} - \alpha\frac{\partial_t v}{v} \leq a\alpha^2 \left( \frac{1}{t} + \frac{K}{p_1^2}\frac{L}{2(\alpha-1)} \right)^2
    \end{equation*}
    on $M\times (0,T]$. In the special case where $K = 0$, we have
    \begin{equation}\label{ab-est-eps-eq}
        \frac{|\nabla v|^2}{v} -  \frac{\partial_t v}{v} \leq \frac{a}{t}
    \end{equation}
    on $M\times (0,T]$.
\end{corollary}
\begin{remark}\rm{
    When $N \in[n, \infty)$, $\varepsilon=1$ and $p_1=p_2=1$, we have $c=\frac{1}{N-1}$, in which case, Theorem \ref{aronson-benilan-epsilon-range} recovers Theorem \ref{aronson-benilan-huang} and Corollary \ref{ab-est-eps} recovers Corollary \ref{aronson-benilan-estimate}.
    The reason why we considered the local Aronson-B\'{e}nilan gradient estimate under variable lower curvature bounds with $\eps$-range is that the Laplacian comparison theorem does not hold under constant lower bounds of the $N$-weighted Ricci curvature with $N < 0$.}
\end{remark}
\section{ Aronson-B\'{e}nilan estimate under non-negative curvature}

In this Section 4, we show the Aronson-B\'{e}nilan estimate under non-negative $N$-weighted Ricci curvature with $N \in \left(-\infty, - \frac{2}{m-1}\right)$ without using the Laplacian comparison theorem.
The following theorem is an alternative approach to Corollary \ref{ab-est-eps} with $K = 0$.
\begin{theorem}\label{ab-est-cpt}
    Let $(M,g,\mu)$ be an $n$-dimensional compact weighted Riemannian manifold, $u$ be a positive solution of \eqref{porus_medium_equation} with $m > 1$ on $M\times [0,T]$ and
     $N \in (-\infty,\frac{-2}{m-1}) \cup [n,+\infty]$. We assume $\Ric_{\psi}^{N} \geq 0$. Then we have
    \begin{equation*}
        \frac{|\nabla v|^2}{v} - \frac{\partial_t v}{v} \leq \frac{a}{t}
    \end{equation*}
    on $M\times (0,T]$, where $v = \frac{m}{m-1}u^{m-1}$ and $a = a(m,N)$ defined in \eqref{a-def}.
\end{theorem}
\begin{proof}
    The case $N\in [n,\infty)$ is obtained by \cite[Theorem 7.6]{lili-2}. 
    We can apply the argument in \cite[Theorem 7.6]{lili-2} to generalize it to weaker curvature bounds $\Ric_\psi^N\geq 0$ with $N < 0$.
    By Lemma \ref{basic_lemma_porus_medium}, we have 
    \begin{eqnarray}
        \mathscr{L}(F_1)&\leq& -\frac{1}{a}[(m-1)\Delta_\psi v]^2 + 2m\langle \nabla v,\nabla F_1 \rangle - 2(m-1)\Ric_\psi^N(\nabla v)\nonumber\\
        &\leq& -\frac{1}{a}[(m-1)\Delta_\psi v]^2  + 2m\langle \nabla v,\nabla F_1 \rangle,\label{eq-for-F_1}
    \end{eqnarray}
    where we used $\Ric_\psi^N\geq 0$ in the second inequality.
    Setting $F = tF_1$, we have 
    \begin{eqnarray}
        \mathscr{L}(F) &=& t\mathscr{L}(F_1) + F_1\nonumber\\
        &\leq& -\frac{1}{a}\frac{F^2}{t} + 2m\langle \nabla F, \nabla v \rangle + \frac{F}{t}.\label{eq-for-F}
    \end{eqnarray}
    We apply the maximum principle to $F$ on $M\times [0,T]$. Let $(x_1,t_1)\in M \times [0,T]$ be a point which attains the maximum of $F$ and we assume $F(x_1,t_1) > 0$ (otherwise the proof is trivial), which implies $t_1 > 0$.
    Since we have
    \begin{equation*}
        \nabla F(x_1,t_1) = 0, \quad \Delta F(x_1,t_1) \leq 0, \quad \partial_t F(x_1,t_1) \geq 0,
    \end{equation*}
    we obtain
    \begin{equation*}
        \mathscr{L}(F)(x_1,t_1) \geq 0.
    \end{equation*}
     Combining this with \eqref{eq-for-F}, we have
    \begin{equation*}
        0\leq -\frac{1}{a}\frac{F(x_1,t_1)^2}{t_1} + \frac{F(x_1,t_1)}{t_1}.
    \end{equation*}
    Therefore, we find
    \begin{equation*}
        F(x,t) \leq F(x_1,t_1) \leq a
    \end{equation*}
    for all $(x,t)\in M\times [0,T]$.
    Since $F_1 = \frac{|\nabla v|^2}{v} - \frac{\partial_t v}{v}$, we obtain the desired inequality.
\end{proof}
\begin{remark}\label{fast-diffusion}\rm{
    Although the Aronson-B\'{e}nilan estimate for the fast diffusion equation (i.e. $0 < m < 1$) is also well investigated in \cite{huang-huang-li,huang}, the method to prove the Aronson-B\'{e}nilan estimate for the fast diffusion equation cannot be applied to $\Ric_\psi^N$ with $N < 0$.
    This could be a natural phenomenon since the admissible range $(-\infty,\frac{-2}{m-1})$ for $N$ degenerates as $m\searrow 1$, as we mentioned in the Introduction. 
}
\end{remark}
\textbf{Acknowledgement}

I would like to express my sincere gratitude to my supervisor Shin-ichi Ohta for his support, encouragement and for providing many fruitful suggestions and comments on preliminary versions of this paper.

\end{document}